\documentclass[12pt]{amsart}
\usepackage{amsmath, amsthm, amscd, amsfonts, amssymb, graphicx, color, mathabx,     enumitem, mathtools, comment,amsrefs}
     
    \makeatletter
    \def\l@subsection{\@tocline{2}{0pt}{2.9pc}{5pc}{}}
    \def\l@subsubsection{\@tocline{2}{0pt}{5pc}{7.5pc}{}}
    \makeatother
\usepackage[margin=2.8cm]{geometry}
\usepackage{ulem}

\usepackage{mathabx}
\usepackage[bookmarksnumbered, colorlinks, plainpages]{hyperref}
    \hypersetup{colorlinks=true, linkcolor=blue, citecolor=magenta, filecolor=magenta, urlcolor=cyan}
\allowdisplaybreaks
    \mathtoolsset{showonlyrefs,showmanualtags} 
\numberwithin{equation}{section}
\newtheorem{theorem}[equation]{Theorem}
\newtheorem{lemma}[equation]{Lemma}

\theoremstyle{definition}

\newtheorem{remark}[equation]{Remark}

\newcommand{\Z}{\mathbb Z}

\begin{document}
\title[Density Vinogradov Theorem]{
A Density Theorem for Higher Order \\ Sums of Prime Numbers
}

\author[Lacey]{Michael T. Lacey} 
    \address{ School of Mathematics, Georgia Institute of Technology, Atlanta GA 30332, USA}
    \email {lacey@math.gatech.edu}
    \thanks{Research of MTL and YR supported in part by grant  from the US National Science Foundation, DMS-2247254.}
\author[Mousavi]{Hamed Mousavi}
    \address{School of Mathematics, University of Bristol, Bristol, UK}
    \email{gj23799@bristol.ac.uk}
\author[Rahimi]{Yaghoub Rahimi}
    \address{School of Mathematics, Georgia Institute of Technology, Atlanta GA 30332, USA}
    \email{yrahimi6@gatech.edu}
\author[Vempati]{Manasa N. Vempati}
    \address{Department of Mathematics, Louisiana State University, USA}
    \email{nvempati@lsu.edu}

\begin{abstract}  
Let $P$ be a subset of the primes of lower density strictly larger than $\frac12$. 
Then, every sufficiently large even integer is a sum of four primes from the set $P$. 
We establish similar results for $k$-summands, with $k\geq 4$, and for $k \geq 4$ distinct subsets of primes.   
This extends the work of H.~Li, H.~Pan, as well as X.~Shao on sums of three primes, and A.~Alsteri and X.~Shao on sums of two primes.  
The primary new contributions come from elementary combinatorial lemmas. 
 \end{abstract}

\maketitle 
\tableofcontents

\section{Introduction}

The work of Vinogradov from the 1930's established that every sufficiently large odd integer is the sum of three primes, with corresponding result for higher order sums of primes.    Our focus here is on density versions of these results.  In the instance of $k$-fold sums of primes, with $k \geq 4$, if one takes a set of primes $P$ of lower density (in the primes) strictly greater than $\frac{1}{2}$, then Vinogradov's result remains true, restricting the representation to sums of $k$ primes from $P$. 

The case of three-fold sums  primes was initiated in work of Li and Pan \cite{MR2661445} who primarily focused on sums from three distinct sets of primes. 
The case of a single set of primes was settled by Shao \cite{MR3165421}.  A bit informally, any subset of primes of lower density strictly larger than $\frac 58$ satisfies the conclusion of Vinogradov's result.  Recently, Alsetri and Shao \cite{alsetriShao} considered the delicate case of the sum of two primes.  The Waring-Goldbach variants have been investigated in \cites{gao,sal}.

 The goal of this paper is to extend a density version of Vinogradov's three prime theorem to 
sums of four and more summands.  
They are stated in terms of the \emph{lower density} of subsets $P$ of the primes $\mathbb P$.  This is defined by 
\begin{equation}  \label{e:lower} 
     d(P) = \liminf_{N \rightarrow \infty} \frac{|P\cap \mathbb P \cap [1,N]|}{| \mathbb P \cap [1,N]|}.
\end{equation}

Our result is that for primes from a  single set,  density $\frac12$ is critical for all  $k$-fold sums,  $k\geq 4$. 

\begin{theorem} \label{t:oneSet} Let $k\geq 4$. 
  Let $P \subseteq \mathbb{P}$ be a subset of primes with the density $d(P) > \frac{1}{2}$. Then there is an integer $N_{P,k}$ so that for every  integer $n > N_{P,k}$ congruent to $k$ mod $2$, there are $p_i \in P$, $1\leq i \leq k$ such that  $n = \sum_{i=1}^k p_i$.
\end{theorem}

In particular, the $\tfrac 58$ obstruction from the $3$ term case is no longer present. 
That is to be expected, as the $3$ term case is the critical case in Vinogradov's Theorem. 
The density condition $\tfrac 12$ is sharp. Consider the set of primes $P = \{p \in \mathbb{P} \colon  p \equiv 1 \operatorname{mod} 3 \}$, then we see that $d(P) = \frac{1}{2}$ and  only integers  equivalent to $k \operatorname{mod} 3$  can be the sum of $k$ primes in $P$.

The case of distinct subsets of primes is 
a bit more involved.  

\begin{theorem} \label{t:kDistinct}
  Let $k \geq 4$ and $P_1,\cdots,P_k \subseteq \mathbb{P}$ be $k$ subsets of primes satisfying  the lower density conditions $ \sum_{i=1}^k d(P_i) > \frac{k+1}{2}$ and 
  each $P_i$ contains at least one odd prime. 
  Then for every sufficiently large integer $n$ congruent to $k$ mod $2$, there are $p_i \in P_i$ such that $n = \sum_{i=1}^k p_i$.
\end{theorem}

Above, some $P_i$ can be finite collections, consisting of only one odd prime. 
This density condition is sharp in two different ways.  
Take $\frac12$ of the primes to be  $ P _{\frac12}= \{p \in \mathbb{P} \colon  p \equiv 1 \operatorname{mod}3 \}$. 
For $k \nequiv 1 \mod 3$,  and 
 $1\leq j \leq k-1$ set $P_j = P _{\frac12}$, 
 and $P_k = \mathbb P \setminus \{3\}$. 
The sum of densities is $\tfrac {k+1}2$,  and  integers that have remainder $k-1$ mod $3$ cannot be of the form $ p_1 + \cdots + p_k$ where $p_i \in P_i$. The construction is easy to modify for $k\equiv 1 \mod 3$. 
A second example is to take $P_j$ to be 
$P _{\frac12}$ for $1\leq j \leq k-2$, 
and then $P _{k-1}= \mathbb{P} $, and $P _{k} = \emptyset $. 
The sum of densities again equals $\frac{k+1}{2}$, yet no integer can be represented as a sum of primes $p_i\in P_i$.

Note that this extension is non-trivial. Consider  $d(P_i)  \approx \frac{k+1}{2k} + \frac{\epsilon}{k^2}$, so $d(P_i) > \frac{1}{2}$ and $ \sum_{i=1}^k d(P_i) > \frac{k+1}{2}$.
 But for every $3 \leq k' < k$ we have that $k'(\frac{k+1}{2k} + \frac{\epsilon}{k^2}) < \frac{k'+1}{2}$. 
Hence, for this setup we can never reduce the problem to less than $k$ primes.

In view of the Roth Theorem on three term arithmetic progressions, the question of a density version of the Vinogradov three primes 
is natural. Closely related questions appear in 
S\'ark\"ozy \cite{MR1832691}*{\S9.3,9.4}. 
It became even more apparent after B.~Green's proof of Roth Theorem in the primes \cite{MR2180408}.  
 In particular, the transference method invented in that paper has had a number of applications, and is a basic tool in the this paper. 
 We point the reader to this useful survey by Prendiville \cite{MR3680137}.

 H.~Li and H.~Pan \cite{MR2661445} established a result for the sum of three distinct sets of primes $P_j$, 
 for $j=1,2,3$. In particular, if $d(P_1)+d(P_2)+d(P_3) >2 $, then every sufficiently large odd integer is a sum of three primes $p_j\in P_j$.  
 And, X.~Shao \cite{MR3165421} determined that $\frac58$ is the critical lower density required for the three term result, with all three primes from a fixed set.    
 There is a local obstruction for all of these results, and that for Shao's Theorem arises from the integers mod $15$: 
 There are $5$ residues of $\mathbb{Z} _{15} ^{\ast}$, whose triple sum is not all of $\mathbb{Z} _{15}$.
 Very recently, Alsetri and X.~Shao \cite{alsetriShao} studied the binary case. 
 In particular, there is no finitary version of the density Goldbach result in this setting.   
 In related developments, D.~Basak \emph{et al.} 
 \cite{basak2024primesneededternarygoldbach} 
 showed that a zero density set of primes $Z$ is sufficient to represent all odd integers $n\geq 7$ 
 as sum of three primes in $Z$.  
 And, Leng and Sawhney \cite{2024arXiv240906894L} showed that certain sets of primes with restricted digits satisfy Vinogradov's three prime Theorem.  

 Our paper follows the lines of proof of \cites{MR2661445,MR3165421}. 
This proof has two principal components, first a set of showing that the local versions of the statements hold.  Second, a transference argument, which is due to Green \cite{MR2180408}.  
These are stated in the next section.  
They in turn follow from certain purely combinatorial lemmas. 
These are proved in \S\ref{sec:combinatorial_lemmas}. 
In contrast to the work of \cites{MR2661445,MR3165421}, 
the proofs are rather elementary.
That is the principal novelty of this paper.  
The second transference argument is  
 sophisticated and delicate argument.  
 In our setting, there is no essential difficulty in adapting it to context.

\section{Key Arithmetic Lemmas} 

For a  single sequence of primes, we need this Lemma. It provides sufficient conditions to cover $\mathbb{Z} _q$ in a $k$-fold sumset of a set in $\mathbb{Z}_q ^\ast$, for square free $q$.  Of course it is important that $k\geq 4$. 

\begin{lemma}   \label{t:ZpOneSet}
     Fix integer $k\geq 4$. 
     Let    $q$ be a  squarefree integer.
     Let $f \colon \Z_q^* \rightarrow [0,1]$ 
     satisfy 
    $$ \label{e.ll}
    \sum_{x \in \Z_q^*} f(x)  > c \phi(q),
    \qquad c> \tfrac12. 
    $$
     Then for any  integer $n$,  there exists $x_1,\cdots,x_k \in \Z_q^*$ such that $n \equiv \sum_{i=1}^k x_i \mod q$ and
    $$ \label{e:ZpOneSet}
     \sum_{i=1}^k f(x_i) >  kc , 
      \quad \textup{and} \quad \prod _{i=1}^k 
      f_i(x_i) > 0.
    $$
\end{lemma}

The requirement that $c>\frac12$ in the hypothesis is sharp,  in view of the local mod $3$ obstruction.  
The important points about the conclusions are   two fold. 
First, the lower bound for the sum in  \eqref{e.ll} is positive and independent of $n$ and $q$. The nature of the bound is important to 
 the proof of the Lemma, as it permits an inductive proof. 
But it is  not important to the Theorem. 
Second, it is important that  the $f_i(x_i)$ are non-zero.

Simple examples show that the 
individual the $f (x_i)$ need not have a uniform bound.  
And, in the proof of the main theorem, we bound the number of representations of a large integer by the principal Vinogradov type bound, times an expression that involves the minimum of the $f(x_i)$.

For distinct subsets of primes, we need this result.    

\begin{lemma}   \label{t:Zqk}
    Let $q $ be a squarefree integer,   $k\geq 4$, and $n$ an integer congruent to $k$ mod $2$.  
    Fix $c > \tfrac {k+1}{2k}$.
     Let $f_i:\Z_q^* \rightarrow [0,1]$ for $i =1,\cdots,k$ be arbitrary functions that are not identically zero, with the average 
    \begin{equation}
    \sum_{x \in \Z_q^*} \sum_{i=1}^k f_i(x) > kc \phi(q).
    \end{equation}
        Then for any $n \in \Z_q$, there exists $x_1,\cdots,x_k \in \Z_q^*$ such that $n = \sum_{i=1}^k x_i$ and
    $$ \label{e:kclower}
    \prod _{i=1}^k  f_i(x_i) > 0, 
     \quad \textup{and} \quad 
          \sum_{i=1}^k f_i(x_i) >  
          \begin{cases}
              4(\tfrac {16}3 c-1)  & k=4,\textup{and $ $}  3\mid q 
              \\ 
              ck   & \textup{otherwise} 
          \end{cases}
$$
\end{lemma}

\subsection{Proof of Lemma \ref{t:Zqk} } 
 \label{sub:proof_of_theorem}

Both Lemmas conclude that $\mathbb{Z} _q$ is contained in a sumset. 
 Thus, the Cauchy-Davenport Theorem is an important tool, in the case that $q$ is prime. 
  We recall it here.  

\begin{theorem}[Cauchy-Davenport Theorem] 
 \label{thm:CDav}
Let  $p$   a prime number and $k \geq 2$ and $A_i$ be non-empty subsets of $\Z/p\Z$ for $1 \leq i \leq k$. Then
$$
|A_1+\cdots+A_k| \geq \min(p,\sum_{i=1}^k |A_i| -k+1).
$$
\end{theorem}

We will give the proof of Lemma \ref{t:Zqk},  
the proof of the other being easier. 
 We begin by proving that if $q$ 
 is an odd square free integer,   with $3\nmid q$,  
 then a stronger conclusion holds.  
 Namely, for any $n\in \mathbb{Z}_q$, there 
 are $x_1, \ldots, x_k \in \mathbb{Z} _q ^{\ast}$ with $n=x_1+ \cdots+ x_k$, and 
 $$ \label{e:stronger}
     \sum_{i=1}^k f_i(x_i) > ck, 
      \quad \textup{and} \quad \prod _{i=1}^k 
      f_i(x_i) > 0.
    $$
The conclusion is stronger than the Theorem as the lower bound on the sum is larger, and we have imposed divisibility criteria. 
This conclusion is suitable for induction.

The proof is by induction on $\omega (q)$, 
the number of prime factors of $q$.  
The base case is that of a prime $p >3$. 
 Set 
    $$\{f_i(m)\}_{m\in \mathbb{Z}_p^{*}}=\{a_{m,i}\}_{m=0}^{p-2}$$
    such that $a_{0,j} \geq a_{1,j} \geq  \cdots \geq a_{p-2,j}$ are the reduced residues. By   assumption,   we know that
    $$
    \sum_{j=1}^k\sum_{m=0}^{p-2}a_{m,j} > kc  (p-1),
    $$
    Using Lemma \ref{lemk}, there exists $(i_1, \ldots, i_k)$ with $i_1+ \cdots+ a_i \geq p-1$ such that 
\begin{equation}\label{e:primecase}
        a_{i_1,1}+  \cdots+ a_{i_k,k} > ck,  
        \qquad  a_{i_1,1} \cdots  a_{i_k,k} \neq 0.
    \end{equation}
    Define $A_m = \{x  \in \Z_p^*, f_m(x) \geq a_{i_m,m} \}$.
     By construction, 
    $
        A_m = \{a_{0,m},a_{1,m},\cdots , a_{i_m,m}\}
    $, 
    which means that $\lvert A_m \rvert = m+1$. Hence, we have 
    \begin{equation}
 \sum_{j=0} ^{k}    \lvert A_{i_j} \rvert 
 \geq \sum_{j=0} ^{k} i_j+1 
 \geq p-1+k. 
    \end{equation}
    Now we use the Cauchy-Davenport  Theorem \ref{thm:CDav}, to conclude that the sumset  
     $A_{a_1}+ \cdots +A_{i_k}$  contains $\mathbb{Z} _p$.  Pick $n\in \mathbb{Z}$.  There exists $x_j\in A_{i_j}$,   such that 
    $$n \equiv  \sum_{j=1}^k x_j \pmod{p}$$
    and 
    $f(x_j) \geq a_{i_j} $. 
    Remember \eqref{e:primecase} to immediately conclude  that 
    \begin{align}
        \sum_{j=1}^k f_j(x_j) > ck , \qquad   
        f(x_1) \cdots f(x_k) \neq 0
    \end{align}
    This is the base case of \eqref{e:stronger}. 

    \smallskip 
    \emph{Induction step.} Suppose that for an integer $\omega > 1$, that \eqref{e:stronger} holds for all  $\omega (q) < \omega $.  Given $q $, odd squarefree integer  not divisible by $3$, and with $\omega (q)=\omega $, write $ q = q_1 p$, where $p$ is a prime.  
    The point here is that $\mathbb{Z} _q = \mathbb{Z} _{q_1} \oplus \mathbb{Z} _p$, 
    and we can identify $\mathbb{Z} _q ^{\ast} $ with $\mathbb{Z} _{q_1} ^{\ast} \times \mathbb{Z} _p ^{\ast}$.    
    And so we will argue fiberwise. 

Define $g_m  \colon \mathbb{Z}^*_{q_1}\rightarrow [0,1]$ by 
    $$
    g_j(x) = \frac{1}{\phi(p)} \sum_{y \in \Z_p^*} f_j(x,y), \qquad x \in \mathbb{Z}  _{q_1} ^{\ast}, \ 1\leq j \leq n.  
    $$
    By induction assumption, for any $n \in \mathbb{Z} $,  there exists 
    $x_1, \cdots,x_k \in \Z_{q_1}^*$ such that $n\equiv x_1+\cdots+x_k\pmod{q_1}$, and 
    $$
    g_1(x_1) + \cdots +g_k(x_k) > c k, 
    \quad\textup{and} \quad 
     \prod _{j=1}^k g(x_k)>0.
    $$
    By definition of the $g_j$, we have 
    $$
    \sum_{y \in \Z_p^*} f_1(x_1,y) + \cdots+ f_k(x_k,y) > ck \phi(p). 
    $$
We conclude from the base case,  that  
for any $m\in \mathbb{Z} _p$, 
there are $y_1, \ldots ,y_k$ so that 
$m \equiv y_1+ \cdots + y_k \mod p $,
\begin{equation}
f_1(x_1,y_1) + \cdots+ f_k(x_k,y_k) > ck
\quad 
\textup{and} 
\quad  \prod _{j=1} ^{k} f(x_j,y_j) >0
\end{equation}
with $m= y_1 + \cdots + y_k$.  
That completes the proof of \eqref{e:stronger}. 

\medskip

We have established Lemma \ref{t:Zqk} for 
odd square free $q$, not divisible by $3$. 
Note that the proof above applies in the case of $k\geq 5$, to include the case of $3\mid q$. 
This is because Lemma \ref{lemk} includes the case $n=2=\lvert \mathbb{Z} _3 ^{\ast} \rvert$ for $k\geq 5$.  
 
In the case of $k=4$, we need to address the case of $q$ odd, squarefree, and $3\mid q$. 
This uses Lemma \ref{l:sharpn2}. Note that $n=2$ in that Lemma, which is the cardinality of $\mathbb{Z} _3 ^{\ast}$.  
Write $q=3q_1$, and we modify in the obvious way the inductive approach.  Namely, we argue fiberwise, since $\mathbb{Z} _{q} = \mathbb{Z} _3 \oplus \mathbb{Z} _q$. 
  We have the condition \eqref{e:stronger}, which is applied to the data 
  \begin{equation}
  g _{j} (x) = \frac{1}{\phi(p)}\sum_{y\in \mathbb{Z} _q ^{\ast}} f_j(x,y), 
  \qquad x\in \mathbb{Z} _{q_1} ^{\ast},\ 
  1\leq j \leq 4.  
  \end{equation}
  Then, the hypothesis for \eqref{e:stronger} are in place. We conclude that 
  for $n\in \mathbb{Z} $, there are $x_1, \ldots, x_k \in \mathbb{Z} _q ^{\ast}$ such that  
  $n \equiv x_1 + \cdots+ x_4 \mod q_1$, 
\begin{equation}
  g_1 (x_1) + \cdot + g _{4} (x_4) 
  > c 4, \qquad  g_1 (x_1) \cdots g _{4} (x_4) >0.
 \end{equation}
  The stronger conclusion then tells us that 
  for the functions $f_j (x_j, y)$, for $y\in \mathbb{Z} _3 ^{\ast}$, the hypotheses of 
  Lemma \ref{l:sharpn2} are in place. 

  And so we can conclude this. 
  For odd squarefree $q$, 
  for any $n\in \mathbb{Z}_ q$, there 
 are $x_1, \cdots, x_k \in \mathbb{Z} _q ^{\ast}$ with $n=x_1+ \cdots+ x_k$, and 
 $$ 
     \sum_{i=1}^k f_i(x_i) > (2c-1)k, 
      \quad \textup{and} \quad \prod _{i=1}^k 
      f_i(x_i) > 0.
    $$

\medskip 

The last step is to prove the Theorem for even squarefree $q$. Thus, write $q=2 q_1$.  
Note that $\mathbb{Z} _q = \mathbb{Z} _2 \oplus \mathbb{Z} _{q_1} $.  
Now,  $n \equiv k \mod 2$.  
Write $n = (n_1, n_2)\in \mathbb{Z} _2 \oplus \mathbb{Z} _{q_1}$. 
The function $f$ is supported on $(1, r)$, for $r\in \mathbb{Z} _{q_1}$.  
The condition \eqref{e:stronger} gives us 
$x_1, \ldots, x_k \in \mathbb{Z} _{q_1} ^{\ast}$ such that  $x_1 + \cdots + x_k \equiv  n_2 \mod q_1$,  
and as well 
\begin{equation}
 f(1, x_1) + \cdots + f(1,x_k) 
 > c k, 
 \quad \textup{and} \quad 
 f(1, x_1) \cdots  f(1,x_k) >0. 
 \end{equation}
But then we have 
\begin{equation}
(1, x_1) + \cdots + (1,x_k) \equiv (n_1,n_2) 
\mod \mathbb{Z} _2 \oplus \mathbb{Z} _{q_1}. 
\end{equation}
And that completes the proof.

\section{Combinatorial Lemmas} 
\label{sec:combinatorial_lemmas}


We state and prove here the combinatorial lemmas central to the previous section.  
This is the main novelty of the paper.

\subsection{A Single Set of Primes} 
\label{sub:a_single_set_of_primes}


The main Lemmas are first for a single sequence, corresponding to a single sequence of primes.  

\begin{lemma} \label{lemOne}
    Fix $c>1/2$. 
    Let $k \geq 4$  be an integer 
    and $n\geq 2$ an even integer.  Let $\{ a_{i}\}_{i =0}^{n-1}$  be a  non-increasing sequence in $[0,1]$. 
    Assume that 
    $\sum_{i=0}^{n-1}  a_{i} > nc$ with $c>\frac12$, then there exist  $0\leq i_1,\cdots,i_k \leq n$ with $\sum_{j=1}^k i_j \geq n$ such that 
    \begin{equation}
        a_{i_1}+\cdots + a_{i_k}  > ck, 
        \quad \textup{and}  \quad 
        a_{i_1} \cdot\cdots \cdot  a_{i_k} >0. 
    \end{equation}
\end{lemma}

Here, we need a statement that is true for all even $n \geq 2$, with $k$ fixed. The choice of $c> \tfrac 12$ is required for $n=2$, which corresponds to the mod $3$ obstruction to the main theorem. 
We remark that the statement is true for all \emph{odd} integers $n\geq2$ as well. 
But we don't need that fact as  the totient function takes even integer values when applied to integers $q\geq 3$.

We turn to the proof, based on an induction on $n$. 
The base case   $n=2$.

\medskip     
    \textit{Base of Induction:} For $n=2$, let $a_0+a_1 > 2c>1$. It forces  $a_0,a_1 \neq 0$. 
    Consider the $k$ tuple of all $1$'s.  
If the conclusion does not hold, then 
\begin{align}
kc  & \geq k a_1 
\\ 
& > k(2c-a_0)\geq k (2c-1). 
\end{align}
We conclude that $1-2c\geq 0$, which is a contradiction.

\smallskip 

\textit{Induction Step:} 
Let $n\geq 4$ is an even integer, and assume that the statement of the Lemma holds for integer  even $m < n$.

Suppose that for some integer $1\leq i \leq n-1$ we have $a_0+a_i \leq 2c$.  Then this condition holds with $i=n$, and we can construct a  a sequence of length $n-2$ by 
removing the first and last entry from $\{a_i\}$, 
namely 
\begin{equation}
\{b_j\} = (a_1, \ldots, a_{n-2}). 
\end{equation}
It satisfies the hypothes of the of the Lemma for $n-2\geq 2$.  Thus, there are $j_1, \ldots, j_k$ so that $j_1+ \cdots + j_{n-2}\geq n$, 
all $b _{j_l}$ are non-negative, and 
\begin{equation}
b _{j_1} + \cdots + b _{j_k} > ck. 
\end{equation}
This completes this case, since $b _{j_k} = a _{j_k+1}$, and $n\geq 4$.  

Thus,  $a_0+a_i > 2c$ for every $1\leq i \leq n-1$. 
The sequence has at least $\lceil nc \rceil$ 
nonzero entries. As $n$ is even, this is at least $n/2+1$, that is $a_{n/2} >0$.  
And $a_0+a_{n/2} > 2c$.   
Form the $k$-term sequence 
\begin{equation}
(0, n/2, 0,n/2, \ldots)
\end{equation}
Since $k\geq 4$, the sum of the term above is at least $n$.  And, 
\begin{equation}
a_0 + a _{n/2} + a_0 + \cdots > ck. 
\end{equation}

    \begin{remark} The examples showing $c=\tfrac12 $ is sharp are easy to construct from sequence $\{a_i\}$ that take values $1$ and $\tfrac12$ according to the parity of $n$ and $k$. 
    
    \end{remark}

\subsection{Four and More Sets of Primes} 
\label{sub:many_sets_of_primes}


For four  sequences, the situation is more complicated.  This is the critical case, when $n=2$.  Its hypotheses are the same as the other cases, but the conclusion is weaker. 

\begin{lemma}\label{l:sharpn2}
    Let $n = 2$, and let $\{ a_i\}_{i =0}^{n-1}$, $\{ b_i\}_{i =0}^{n-1}$, $\{ c_i\}_{i =0}^{n-1}$, and $\{ d_i\}_{i =0}^{n-1}$ be four decreasing sequences in $[0,1]$, and $a_0b_0c_0d_0 > 0$. Assume that $\sum_{i=0}^{n-1} (a_i+b_i+c_i+d_i) > 4cn$ with $c>\frac{5}{8}$, then there exists  $(i,j,k,l)$ with $i+j+k+l \geq n$ such that 
    \begin{equation}
        a_i+ b_j + c_k + d_l > \tfrac{16}{3}c - 1,
    \end{equation}
    where $a_ib_jc_kd_l > 0$.
\end{lemma}

The $\tfrac{16}{3}c - 1$ is sharp.  
This example suffices.  
\begin{align}
a_{0,j}&=1, \ j=1,2,3,    \quad a _{0,4}= \epsilon , 
\\ 
a_{1,j}&=\tfrac 23, \ j=1,2,3,    \quad a _{0,4}= 0 .  
\end{align} 

\begin{proof}
The hypothesis written out is 
        \begin{align}\label{e:n4diff}
             a_0+a_1 + b_0+b_1 + c_0+c_1+d_0+d_1 > 8c >5.
        \end{align}
    There are at least $6$ nonzero elements. 
        
    Without loss of generality, assume that $a_1,b_1 \neq 0$. There is nothing more to prove if $a_1 + b_1 + c_0 + d_0 > \frac{16c}{3}-1$. So assume otherwise, which immediately implies that  $a_0+b_0+c_1+d_1 > \frac{8c}{3}+1> \frac{16c}{3}-1$. Considering $c_1,d_1 = 0$ implies that  $a_0+b_0>7/3$. But this is impossible,  so $a_1b_1c_1> 0$, which means that there are seven nonzero elements. 

    So for every $c>5/8$, without loss of generality, we can assume that $a_1b_1c_1>0$.

    If $d_1>0$, then either 
    $$a_0+b_0+c_1+d_1>\frac{16c}{3}-1,$$
    or 
    $$a_1+b_1+c_0+d_0>\frac{16c}{3}-1.$$
    So assume that $d_1=0$. 
    
    Take $y=\frac{1}{3}(16c-(a_0+b_0+c_0)-2d_0)$. If $y+d_0< \frac{16c}{3}-1$, then  $a_0+b_0+c_0> 3$, which is impossible.
    So $y+d_0>\frac{16c}{3}-1$. Now we are done if at least one of the following happens
        \begin{align}
            &a_0 + b_1 + c_1 +d_0 > y+d_0\\
            &a_1 + b_0 + c_1 +d_0> y+d_0\\
            &a_1 + b_1 + c_0 +d_0> y +d_0
        \end{align}
        Assume otherwise, that $a_0 + b_1 + c_1 \leq y$, $a_1 + b_0 + c_1 \leq y$, and $a_1 + b_1 + c_0 \leq y$. Adding them up implies that $(a_0+b_0+c_0) + 2(a_1+b_1 + c_1) \leq 3y$. So we have
        \begin{align}
            2\times \eqref{e:n4diff} - (a_0+b_0+c_0) - 2(a_1+b_1 + c_1) > 16c- 3y 
        \end{align}
        which means that $(a_0+b_0+c_0) > 2(8c-d_0)- 3y> a_0+b_0+c_0$. It completes the proof. 
\end{proof}

The remaining cases present no extra complications in the conclusions.

\begin{lemma} \label{lemk}
    Fix $c>\frac{k+1}{2k}$. 
    For integers $n,k$ assume 
     $(n \geq 3, k \geq4)$ or $(n = 2, k\geq 5)$.  
    Let $\{ a_{i,j}\}_{i =0}^{n-1}$ for $1 \leq j \leq k$ be $k$ non-increasing sequences in $[0,1]$, with $a _{0,j}>0$ for all $1\leq j \leq k$. 
    Assume that $ 
    \sum_{i=0}^{n-1} \sum_{j=1}^k a_{i,j} > c nk $ then there exist  $i_1,\cdots,i_k$ with $\sum_{j=1}^k i_j \geq n$ such that 
    \begin{equation} \label{e:715}
        a_{i_1,1}+\cdots + a_{i_k,k} 
         >  ck ,
    \quad \textup{and} \quad 
    a_{i_1,1} \cdots a _{i_k,k} \neq 0. 
    \end{equation}
\end{lemma}
 
The value of $c$ in the hypothesis is sharp for all values of $k$, with $n=2$. 
We state and prove this lemma for all integers $n$, as there is no simplifications we could find by restricting to even values of $n$. \textbf{}

\begin{proof}
The proof is by induction on $n$. 

\smallskip 
\textbf{Bases Cases:} 
There are two base cases. 
The first  is $n=2$, and $k\geq 5$.
    Assume that 
    $$
    (a_{0,1}+a_{1,1}) + \cdots + (a_{0,k}+a_{1,k}) > 2kc,
    $$
    where $c > \frac{k+1}{2k}$ and $a_{0,j} \neq 0$ for all $1 \leq j \leq k$. Since $2kc > k+1$, we must have at least $k+2$ non-zero $a_{i,j}$. 
    All of the $a_{0,j}$ are non-negative, so  $a_{1,j}$ is non-zero for at least two choices of $j$.

    By reordering, we can assume that $a _{1,j}$ is decreasing in $j$.  Thus, $a _{1,2} >0$.  If $a _{1,3}=0$, we have 
    \begin{equation} \label{e:return}
    a_{1,1} + a_{1,2} + a_{0,3} + \cdots + a_{0,k} > 2kc -2 > kc. 
    \end{equation}
    And, we complete the proof as all the numbers above are positive.    
    If this inequality fails, then the complementary sum 
    \begin{equation}
    a_{0,1} + a_{0,2} + a_{1,3} + \cdots + a_{1,k}  > kc. 
    \end{equation}
    The first two terms are positive by assumption, and $kc > 3$, so there must be $4$ non-zero numbers.  From those, we can complete the proof. 

\smallskip 

The second base case is $n=3$ and $k=4$, 
in which case $12c > \frac {15}2$. 
Assume that 
        $$
        a_0+a_1+a_2 + b_0+b_1+b_2 + c_0+c_1+c_2+d_0+d_1+d_2 > 12c.
        $$
    Note that $12c>6$, so at least seven terms in the above sum are non-zero. Obviously, we know that $a_0,b_0,c_0,d_0\neq 0$ by assumption. We claim that $b_1\neq 0$. If not,  then $b_2=c_1=c_2=d_1=d_2=0$ as well by construction. That is, there are only six non-zero terms when 
     we must have at least seven nonzero terms. It immediately means that $a_1\neq 0$ as $s_1\geq s_2$. 
    
    We have several cases:
    \begin{itemize}
        \item Assume that $a_2 = 0$, then $b_2,c_2,d_2 = 0$ by the fact that $s_1\geq s_2\geq s_3\geq s_4$. It forces $c_1$ to be nonzero. Pick the $4-$tuple $(a_1,b_1,c_1,d_0)$ are nonzero and 
        $$a_1+b_1+c_1+d_0 > 12c- (a_0+b_0+c_0+d_1)> 12c-4 > 4c.$$
        So from now on assume that $a_2\neq 0$.
        \item If $c_1 = 0$, then $c_2,d_1,d_2 = 0$. It implies that 
        $$a_2+b_1+c_0+d_0 > 12c-(a_0+a_1+b_0+b_2)> 12c-4 > 4c.$$ 
        So from now on assume that $c_1\neq 0$.
        \item Assume that  $b_2 = d_1 = 0$, then $c_2,d_2=0$. So 
        $$a_1+b_1+c_1+d_0 > 12c-(a_0+b_0+c_0+a_2) > 12c-4 > 4c.$$ 
        From now on $(b_2,d_1)\neq (0,0)$.
        \item Assume that $b_2 \neq 0$ and $d_1 = 0$, which means that $d_2=0$. So 
        $$
        a_0+a_1+a_2 + b_0+b_1+b_2 + c_0+c_1+ \bar{c}_2 > 12c,
        $$
        where $\bar{c}_2 = c_2+d_0$  and all the terms are non-zero. Rename $\bar c_0 =c_0$ and $\bar c_1 =c_1$. By using pigeon hole principle, there exists $\{i,j,k\} \in \{0,1,2 \}^3$ such that $a_i+b_j+\bar c_k > 4c$. If $k\in \{0,1\}$, we are done. 
        
        Assume that $k=2$. The proof is finished if $c_2 \neq 0$. 
        So assume that $c_2  = 0$. It means that 
         $$
        a_0+a_1+a_2 + b_0+b_1+b_2 + c_0+c_1+d_0> 12c.
        $$
        Look at the subsets 
        $$\{a_1,b_2,c_0,d_0\} \textup{ and } \{a_2,b_1,c_1,d_0\}$$
        If the sum in each set is less than $4c$, then 
        $$
        a_0+b_0+8c-d_0> 12c.
        $$
        So $a_0+b_0> 4c+d_0\geq 2$, which is impossible. It means that $c_2\neq 0$. 
        So from now on $d_1\neq 0$.
        \item Assume that $b_2 = 0$. Then $c_2,d_2=0$. So
        $$
        a_0+a_1+a_2 + b_0+b_1 + c_0+c_1+ d_0+d_1 > 12c.
        $$
        We claim by pigeon hole principle, we have either
         $a_2+b_1+c_0+d_0 > 4c$ or $a_1+b_0+c_1+d_1 > 4c$. Otherwise
         $$a_0> 12c- (a_2+b_1+c_0+d_0) - (a_1+b_0+c_1+d_1) >4c >2$$
         which is a contradiction. So $b_2\neq 0$.

        \item If $a_0+b_2+c_2+d_2 > 4c$, then by the fact that the sequences are decreasing 
        $$a_0+b_2+c_1+d_1>4c$$
        So we are done. Assume otherwise,  $a_0+b_2+c_2+d_2 \leq 4c$. It implies that
        $$
        a_1+  a_2 + b_0+b_1 + c_0+c_1+d_0+d_1 > 8c,
        $$
        So obviously the sum of elements in one of the following subsets is bigger than $4c$: 
        $$\{a_1,b_1,c_0,d_1\} \textup{ and } \{a_2,b_0,c_1,d_0\}$$
        This completes the base of induction.
        
        
    \end{itemize}

\bigskip     

\textit{Induction Step:} 
We assume that $k$ and $n$ are such that 
if $k=4$, then $n>3$, otherwise $n>2$, 
and that the statement of the Lemma holds for $n-1$.   
First assume that there are $l_2,\cdots,l_k > 0$ such that 
    \begin{align}\label{e:0rest}
        a_{0,1} + a_{l_2,2}+ \cdots + a_{l_k,k} \leq  kc
    \end{align}
    We eliminate these terms from our sequence, to reduce back to the induction hypothesis. 
    \begin{align}
        &\{b_{i,1}\} = \{a_{1,1},a_{2,1}\cdots,a_{n-1,1} \}\\
        &\{b_{i,j}\} = \{a_{0,j},a_{1,j}\cdots,a_{l_j-1,j},a_{l_j+1,j},\cdots ,a_{n-1,j} \} \text{ for } 2\leq j\leq k
    \end{align}
    For the matrix $\{b_{i,j}\}$ we have
    \begin{align}
        &\sum_{0\leq i\leq n-2}\sum_{1\leq j\leq k}b_{i,j} \\
        &\quad = \sum_{0\leq i\leq n-1}\sum_{1\leq j\leq k}a_{i,j} - \left(a_{0,1} + a_{l_2,2}+ \cdots + a_{l_k,k}\right) \\
        &\qquad > kcn- kc = kc \cdot (n-1). 
    \end{align}
    Note that $a_{1,1}\neq 0$, otherwise $s_1=0$, which by our assumption implies that $s_2=\cdots = s_k=0$. So the only non zero elements would be $a_{0,j}$ for $j\leq k$, but this immediately means that 
    $$k+1<kcn<\sum_{i,j}a_{i,j}\leq k $$
    which is impossible. So $b_{0,1}=a_{1,1}\neq 0$. 
    In other words, $b_{0,1},\cdots , b_{0,k}>0$. Now all the criteria in the induction step is satisfied, so we conclude that there exists   $u_1,\cdots,u_k$ with $\sum_{j=1}^k u_j \geq n-1$ such that 
    \begin{equation}
        b_{u_1,1}+\cdots + b_{u_k,k}  > \tfrac 7{15} kc,
    \end{equation}
    where $b_{u_j,j} \neq 0$. By a change of variable, we know condition $\sum_{j=1}^k u_j \geq n-1$ for $b_{i,j}$ implies that condition  $\sum_{j=1}^k i_j \geq n$ for $a_{i,j}$ holds. This completes the proof. 

\medskip

    Finally, assume that there are no indices $l_2,\cdots,l_k > 0$ such that equation \eqref{e:0rest} holds. We must have $s_2 \geq 1$.  If $s_3< n-1$, then we must have 
    \begin{align}
        2 &\geq a_{0,1}+a_{s_2,2} +0+\cdots +0
        \\
        &=a_{0,1}+a_{s_2,2}+a_{s_3+1,3}\cdots +a_{s_k+1,k} >kc> \tfrac{k+1}2> \tfrac 52. 
    \end{align}
        which is impossibles. So $s_3=n-1$, 
        as are all the $s_j$ for $3\leq j \leq k$.  Moreover,  
        \begin{align}
            a_{0,1}+ a_{s_2,2}+a_{n-1,3}+a_{n-1,4}+\cdots + a_{n-1,k} > kc. 
        \end{align}
All the summands above are non-negative, so that completes the proof. 

\end{proof}

\section{Proof of the theorems} \label{s:proof_of_theorem}

We prove Theorem  \ref{t:kDistinct}. 
The hypotheses allow some sets of primes $P_i$ to have density $0$, or even be finite, but still contain an odd integer.  We address this case here by induction on $k$, assuming that the Theorem is true for all sets of primes of \emph{positive} density. 
The base case is $k=4$. 
If e.g.~$P_4$ has density zero, then $d(P_1)+d(P_2)+d(P_3) > 5/2 >2$. By the result of Li and Pan \cite{MR2661445}, every sufficiently large odd integer $n$ is the sum of $p_i\in P_i$, for $1\leq i \leq 3$.  But $P_4$ contains an odd prime, so every sufficiently large even integer $m$ is the sum of $p_i\in P_i$, for $1\leq i \leq 3$.  This completes the base case, and the inductive argument is very similar.

We proceed under the assumption that each subset of the primes has positive lower density. This proof follows the the lines of Green \cite{MR2180408}, with the additional modifications of \cites{MR3165421,MR2661445}. That is, the argument is already well documented in the literature.  The proof of Theorem \ref{t:oneSet} is left to the reader.

Let $ k \geq 4$ and 
let $\kappa = 10^{-k}(\sum_{i=1}^k d(P_i)-\frac{k+1}{2})$ and $\alpha_i = d(P_i)/(1+2\kappa)$. 
The integer $n $, congruent to $k$ mod $2$ is the one we show has an ample number of representations as a sum of primes $p_i \in P_i$, for $1\leq i \leq k$. 
This is  subject to the condition that $n$ 
 be sufficiently large.  
Being sufficiently large will consist of a few conditions. The first of these is that  for $\gamma_k = \frac{2}{k}$, this condition holds. 
$$
|P_i  \cap [0,\gamma_k n)| \geq (1+\kappa)\alpha_i \frac{\gamma_k n}{\log(n)},
$$
This is possible by the definition of the lower density \eqref{e:lower}.  

A $W$-trick is useful.  
Consider $\omega(n) = \frac{1}{4}  \log\log(n)$ and $
W = \prod_{\substack{ p \leq \omega(n) \\ p \text{ prime}}} p$. 
So $W \approx \log(n)$ and 
\begin{align}
    \sum_{\substack{x \leq \gamma_k n \\ (x,W) = 1}} \mathbf{1} _{P_i}(x) \log(x) & \geq \sum_{n^{\frac{1}{1+\kappa/2}} \leq x \leq \gamma_k n} \mathbf{1} _{P_i}(x) \log(x)
    \\ & \geq 
    \frac{\log(n)}{1 + \kappa/2}\left( \frac{(1+\kappa)\alpha_i(\gamma_k n)}{\log(n)} - n^{\frac{1}{1+\kappa/2}} \right)
    \\ & \geq 
     \alpha_i \gamma_k n.
\end{align}
Define 
$$
f_i(b) = \max \Biggl\{ 0, \frac{\phi(W)}{\gamma_k n} \sum_{\substack{x \leq \gamma_k n \\ x \equiv b \mod  W}} \mathbf{1} _{P_i}(x) \log(x) - \kappa \Biggr\}, \qquad b\in \mathbb Z^*_W. 
$$
We additionally require $n$ to be large enough that each $f_i(b) \leq 1$, 
and that 
\begin{align}
     \sum_{b \in \mathbb Z^*_W} \sum_{i=1}^k f_i(b) 
    & \geq 
   \frac{\phi(W)}{\gamma_k n} \sum_{b \in \mathbb Z^*_W} \sum_{\substack{x \leq \gamma_k n \\ x \equiv b \mod  W}} \biggl[ \sum_{i=1}^k \mathbf{1} _{P_i}(x) \biggr] \log(x) - k\kappa \phi(W)
   \\  \label{e:SW}
    & \geq 
    \Big(\sum_{i=1}^k \alpha_i - k \kappa \Bigr) \phi(W)
    > \frac{k+1}{2}\phi(W).
\end{align}
This is possible by  the Siegel-Walfisz Theorem, which we need as $W $ is a slowly growing function on $n$.  

 Our Theorem \ref{t:Zqk} assures us  there exist $b_1,\cdots, b_k \in \mathbb Z^*_W$  such that $n \equiv b_1 + \cdots + b_k \pmod  W$ with $f_1(b_1) + \cdots + f_k(b_k) > \frac{k+1}{2}$ and $f_i(b_i) > 0$. Without loss of generality assume that $1 \leq b_i < W$.

Let $N$ be a prime where $(1 + \kappa)n/W \leq N \leq (1 + 2 \kappa)n/W$, 
which is possible by the Prime Number Theorem.  And set 
\begin{gather}
    n' = W^{-1} (n - b_1-b_2- \cdots - b_k) \in \mathbb N , 
    \\ 
 A_i = \{ x \colon  Wx+b_i \in P_i \cap [1,\gamma_k n] \}.
 \end{gather}
We show that $n' \in A_1 + \cdots + A_k$. 
And do so in the typical fashion, namely estimating  
\begin{align}
     \mathbf{1}_{A_1} \ast \cdots 
     \ast \mathbf{1}_{A_k} (n') 
\end{align}
from below, by Fourier methods. Above, $f\ast g$ denotes convolution $\mathbb Z_N$.  Note that representing $n'$ as a sum of elements of $A_i$ in $\mathbb Z_N$ lifts to a solution in $\mathbb N$.   
The rest of the analysis is done on the group $\mathbb Z_N$.  

But we do this with the usual logarithmic correction.  Define a logarithmically weighted average of $A_i$ by 
$$
\alpha_i' = \sum_{x  } \mathbf{1}_{A_i}(x) \lambda_{b_i,W,N}(x),$$
where 
$$
\lambda_{b,W,N}(x) = \begin{cases}
    \phi(W)\log(Wx+b)/WN & \text{if } x \leq N, Wx+b \text{ is a prime.}
    \\
    0 & \text{otherwise.}
\end{cases}
$$
Each $\alpha_i'$ is positive, since 
\begin{align}
    \alpha_i' &= \frac{\phi(W)}{WN} \sum_{x} \mathbf{1}_{A_i}(x) \log (Wx+b_i)
    \\ &\geq  \frac{f_i(b_i)+\kappa}{(1 + 2 \kappa)}\geq \gamma_k  \kappa/2 = \kappa/k. 
\end{align}
Collectively, the satisfy 
\begin{align}
    \alpha_1'+ \cdots+ \alpha_k'  &\geq \frac{\gamma_k}{1 + 2 \kappa} (f_1(b_1) + \cdots + f_k(b_k)+k  \kappa) 
    \\  \label{e:alphaPrime}
    &\geq  \frac{\gamma_k}{1 + 2 \kappa} (\frac{k+1}{2}+k \kappa) \geq 1 + \frac{1}{k} + \kappa.
\end{align}
We consider $A_i$ as subsets of $\mathbb Z_N$. Note that if $n' \in A_1 + \cdots + A_k$ in $\mathbb Z_N$ then $n' \in A_1 + \cdots + A_k$ in $\mathbb Z$, since there are no $x_i \in A_i$ such that $x_1+\cdots + x_k = n' + N$. 
The rest of the analysis is done on this finite group. 

Consider
$$
a_i(x) = \mathbf{1}_{A_i}(x) \mu_i(x), \ \ \mu_i(x) = \lambda_{b_i,W,N}(x).
$$ 
Let $e(x) := e^{2 \pi i x}$.
For a function $f:\Z_N \to \mathbb{C}$, we can define the Fourier transform on $\mathbb Z_N$ by 
\begin{align} \label{e:FDq}
    \tilde{f}(x) = \sum_{n\in \mathbb Z_N} f(n) e\bigl(- xr/N\bigr),
\end{align}
In addition, the convolution on $\mathbb Z_N$ is 
\begin{align} 
     (f * g) \, (x) := \sum_{y\in \mathbb Z_N} f(y) g(x-y).
\end{align}
Therefore we have that $\widetilde{f*g} = \tilde{f} \cdot \tilde{g}$. 

Let $\epsilon, \delta > 0$ and define sets  
\begin{align} 
     R_i &= \{ r \in \mathbb Z_N \ \ | \ \ |\tilde a_i(r)| \geq \delta \},
\\ 
B_i &= \{ x \in \mathbb Z_N \ \ | \ \ \| xr/N \| \leq \epsilon \ \ \text{ for all } \ \ r \in R_i \},
\end{align}
where $\|x \| = \min_{z \in \mathbb Z} |x-z|$.  
So, $R_i$ is a super-level set for the Fourier transform of $a_i$, and $B_i$ is $R_i$ intersected with a Bohr set.  
Then, set 
$\beta_i = \frac{1}{|B_i|} 1_{B_i}$ and $a_i' = a_i*\beta_i*\beta_i$.

\begin{lemma}
\label{l:aa'}
We have the estimate 
\begin{align} 
    \left| \sum _{\substack{x_1,\cdots,x_k \in \mathbb Z_N \\ x_1 + \cdots + x_k = n'    }} \prod_{i=1}^k a_i'(x_i) -  \sum _{\substack{x_1,\cdots,x_k \in \mathbb Z_N \\ x_1 + \cdots + x_k = n'    }} \prod_{i=1}^k a_i(x_i) \right| \ll 
  \frac{C}{N} (\epsilon^2 \delta^{-5/2} + \delta^{k/(k+1)}).  
\end{align}
    
\end{lemma}

\begin{proof}
We have that 
    \begin{align}
        \sum _{\substack{x_1,\cdots,x_k \in \mathbb Z_N \\ x_1 + \cdots + x_k = n'    }} \prod_{i=1}^k a_i(x_i) = (a_1*\cdots*a_k)(n') =\frac{1}{N} \sum_{r \in \Z_N} \prod_{i=1}^k \tilde{a}_i(r)  e\bigl(n'r/N\bigr).
    \end{align}
    Therefore
    \begin{align}
        & \left| \sum _{\substack{x_1,\cdots,x_k \in \mathbb Z_N \\ x_1 + \cdots + x_k = n'    }} \prod_{i=1}^k a_i'(x_i) -  \sum _{\substack{x_1,\cdots,x_k \in \mathbb Z_N \\ x_1 + \cdots + x_k = n'    }} \prod_{i=1}^k a_i(x_i) \right|
        \\ &
        = \frac{1}{N} \sum_{r \in \Z_q}  \left(\prod_{i=1}^k \tilde{a}_i(r)  \right) \left(1 - \prod_{i=1}^k\tilde \beta_i^2(r) \right)e\bigl(n'r/N\bigr).
    \end{align}
    We have that $|R_i| \leq C_1 \delta^{-5/2}$ for some constant $C_1$ and if $r \in R:= R_1\cap \cdots R_k $ then 
    $$
    \left|1 - \prod_{i=1}^k\tilde \beta_i^2(r) \right| \leq C_2 \epsilon^2.
    $$
    Therefore since $|\tilde a_i(r)| \leq 1$,
    \begin{align}
        \left| \sum_{r \in R}  \left(\prod_{i=1}^k \tilde{a}_i(r)  \right) \left(1 - \prod_{i=1}^k\tilde \beta_i^2(r) \right)e\bigl(n'r/N\bigr) \right| \leq  C_1C_2\epsilon^2\delta^{-5/2}.
    \end{align}
    Additionally since $|\tilde \beta_i(r)| \leq 1 $, using Holder's inequality we have 
    \begin{align}
        &
        \left| \sum_{r \not\in R}  \left(\prod_{i=1}^k \tilde{a}_i(r)  \right) \left(1 - \prod_{i=1}^k\tilde \beta_i^2(r) \right)e\bigl(n'r/N\bigr) \right| 
        \\ &
        \leq \sup_{r \not\in R} \left|\prod_{i=1}^k \tilde{a}_i(r)  \right|^{\frac{1}{k+1}} \sum_{r \not\in R}  \prod_{i=1}^k \left| \tilde{a}_i(r) \right|^{\frac{k}{k+1}}
        \\ &
        \leq \delta^{k/(k+1)} \prod_{i=1}^k 
  \left(\sum_{r \not\in R} \left| \tilde{a}_i(r) \right|^{\frac{k^2}{k+1}} \right)^{1/k}
  \\ & \leq C_3 \delta^{k/(k+1)}.
    \end{align}
\end{proof}

This is the same as \cite{MR2180408}*{Lemma 6.3}. It shows that $a'_i$ does not take values much larger than $1$. 

\begin{lemma}
    Suppose that $\epsilon^{|R_i|} \geq C\log\log w/w$. Then for each $x \in \Z_N$
    $$
    |a_i'(x)| \leq (1+2C^{-1})/N.
    $$
\end{lemma}

We also need a variant of Varnavides, but for sumsets. Namely, we need a result that assures us that ample subsets of $\mathbb Z_N$ contain all of $\mathbb Z_N$ in their sumset.  This is \cite{MR2661445}*{Lemma 3.3}, so we omit the proof. 

\begin{lemma} \label{l:V} 
Suppose $k\geq 2$, and $0< \theta_1, \dotsc, \theta_k\leq 1$ satisfy $\theta_1+ \cdots + \theta_k >1$. Set 
\begin{equation}
    \theta= \min\{ \theta_1,\dotsc,\theta_k, (\theta_1+\cdots+\theta_k -1)/(3k-5)\}. 
\end{equation}
Let $N > 2 \theta^{-2}$ be an integer, and $X_1,\dotsc,X_k \subset \mathbb Z_N$ 
satisfy $\lvert X_i\rvert \geq \theta_iN$, for $1\leq i \leq k$.   Then, we have, for every $n\in \mathbb Z_N$, 
\begin{align} \label{e:nu}
 \nu (n) &\coloneqq  \lvert 
 \{ (x_1,\dotsc,x_k) \colon 
 x_i\in X_i,\ n=x_1+\cdots+x_k\}\rvert 
 \\ 
 & \geq \theta^{2k-3}N ^{k-1}. 
\end{align}
\end{lemma}

We can now conclude the proof of our Theorem.  Namely, we show that 
\begin{align} 
      \sum _{\substack{x_1,\cdots,x_k \in \mathbb Z_N \\ x_1 + \cdots + x_k = n'    }} \prod_{i=1}^k a'_i(x_i) \geq \frac{\kappa^{4k-3}}{k^{3k-3}N}.  
\end{align} 
We want to use the previous Lemma, which applies to sumsets. Thus, we define  $A_i' = \{ x \in \Z_N \mid a_i'(x) \geq \alpha_i' \kappa/N \}$.  We have that 
\begin{align}
    \alpha_i' = \sum_{x \in \Z_N} a_i(x) = \sum_{x \in \Z_N} a_i'(x) \leq \frac{1+\kappa}{N}|A_i'| + \frac{\alpha_i' \kappa}{N}(N-|A_i'|),
\end{align}
From this, we see that $A_i'$ has a lower bound on its size: 
$$
 \theta_i \coloneqq|A_i'| \geq \frac{\alpha_i'(1-\kappa)}{1+\kappa}N.
$$
By \eqref{e:alphaPrime}, we see that $A_i'$ satisfy the hypotheses of Lemma \ref{l:V}.  
It follows from \eqref{e:nu}, that $n'$ 
can be represented as a sum of elements of the $A_i'$ at least $\theta^{2k-3}N^{k-1} $ times, where $\theta= \kappa/k$.   
Therefore
\begin{align} 
      \sum _{\substack{x_1,\cdots,x_k \in \mathbb Z_N \\ x_1 + \cdots + x_k = n'    }} \prod_{i=1}^k a'_i(x_i) &\geq \sum _{\substack{x_i \in A_i' \\ x_1 + \cdots + x_k = n'    }} \prod_{i=1}^k a'_i(x_i) 
      \\ 
      & \geq   \theta^{2k-3}N^{-1}\prod_{i=1}^k {\alpha_i' \kappa}
      \\ 
      & = \kappa^{2k} k^{3-2k} N ^{-1}. 
\end{align}
This is the estiamte for $a_i'$. 
Recalling Lemma \ref{l:aa'}, the last step is to choose $0<\delta<1$, and then $0<\epsilon<1$, so that we conclude that 
\begin{align} 
      \sum _{\substack{x_1,\cdots,x_k \in \mathbb Z_N \\ x_1 + \cdots + x_k = n'    }} \prod_{i=1}^k a_i(x_i) &\geq \tfrac 12 \kappa^{2k} k^{3-2k} N ^{-1}. 
\end{align}
And this concludes the proof. \textbf{}


\begin{bibdiv}

    \begin{biblist}
\bib{alsetriShao}{article}{
  title={Density versions of the binary Goldbach problem},
  author={Alsetri, Ali},
  author={Shao, Xuancheng},
  year={2024},
  eprint={2405.18576},
  url={https://arxiv.org/abs/2405.18576},
}

\bib{gao}{article}{
  title={A density version of Waring-Goldbach problem},
  author={Gao, Meng},
  journal={arXiv preprint arXiv:2308.06286},
  year={2023},
}

\bib{sal}{article}{
  author={Salmensuu, Juho},
  title={A density version of Waring's problem},
  journal={Acta Arith.},
  volume={199},
  date={2021},
  number={4},
  pages={383--412},
  issn={0065-1036},
  review={\MR {4309845}},
  doi={10.4064/aa200601-1-2},
}

\bib{MR2180408}{article}{
  author={Green, Ben},
  title={Roth's theorem in the primes},
  journal={Ann. of Math. (2)},
  volume={161},
  date={2005},
  number={3},
  pages={1609--1636},
  issn={0003-486X},
  review={\MR {2180408}},
  doi={10.4007/annals.2005.161.1609},
}

\bib{MR2661445}{article}{
  author={Li, Hongze},
  author={Pan, Hao},
  title={A density version of Vinogradov's three primes theorem},
  journal={Forum Math.},
  volume={22},
  date={2010},
  number={4},
  pages={699--714},
  issn={0933-7741},
  review={\MR {2661445}},
  doi={10.1515/FORUM.2010.039},
}

\bib{gao}{article}{
  title={A density version of Waring-Goldbach problem},
  author={Gao, Meng},
  journal={arXiv preprint arXiv:2308.06286},
  year={2023},
}

\bib{sal}{article}{
  title={A Density version of Waring's problem},
  author={Salmensuu, Juho},
  journal={arXiv preprint arXiv:2003.04918},
  year={2020},
}

\bib{MR3680137}{article}{
  author={Prendiville, Sean},
  title={Four variants of the Fourier-analytic transference principle},
  journal={Online J. Anal. Comb.},
  number={12},
  date={2017},
  pages={Paper No. 5, 25},
  review={\MR {3680137}},
}

\bib{MR1832691}{article}{
  author={S\'ark\"ozy, Andr\'as},
  title={Unsolved problems in number theory},
  journal={Period. Math. Hungar.},
  volume={42},
  date={2001},
  number={1-2},
  pages={17--35},
  issn={0031-5303},
  review={\MR {1832691}},
  doi={10.1023/A:1015236305093},
}

\bib{MR3165421}{article}{
  author={Shao, Xuancheng},
  title={A density version of the Vinogradov three primes theorem},
  journal={Duke Math. J.},
  volume={163},
  date={2014},
  number={3},
  pages={489--512},
  issn={0012-7094},
  review={\MR {3165421}},
  doi={10.1215/00127094-2410176},
}

\bib{2024arXiv240906894L}{article}{
  author={{Leng}, James and {Sawhney}, Mehtaab},
  title={Vinogradov's theorem for primes with restricted digits},
  doi={10.48550/arXiv.2409.06894},
}

\bib{basak2024primesneededternarygoldbach}{article}{
  title={Almost all primes are not needed in Ternary Goldbach},
  author={Debmalya Basak and Raghavendra N. Bhat and Anji Dong and Alexandru Zaharescu},
  year={2024},
  eprint={2409.08968},
  archiveprefix={arXiv},
  primaryclass={math.NT},
  url={https://arxiv.org/abs/2409.08968},
}


\end{biblist} 

\end{bibdiv}
\end{document}